\documentclass[a4paper,11pt]{article}
\usepackage[notref,notcite]{showkeys}
\usepackage{mathrsfs}
\usepackage{latexsym}
\usepackage{amsmath,amssymb}
\usepackage[pdftex]{hyperref}
\usepackage{amsthm}
\usepackage{amsfonts,cite}
\usepackage[usenames]{color}
\usepackage{amssymb}
\usepackage{graphicx}
\usepackage{amsmath}
\usepackage{amsfonts}
\usepackage{amsthm}
\usepackage{mathrsfs}
\usepackage{dsfont}
\usepackage{indentfirst}

\newtheoremstyle{mythm}{1.5ex plus 1ex minus .2ex}{1.5ex plus 1ex
minus .2ex}{\kai}{\parindent}{\song\bfseries}{}{1em}{}
\numberwithin{equation}{section}

\newtheorem{theorem}{Theorem}[section]
\newtheorem{lemma}{Lemma}[section]

\newtheorem{remark}{Remark} [section]

\allowdisplaybreaks[4]
\begin{document}
\title{{\textbf{Analyticity of the solutions to degenerate Monge-Amp{\`e}re equations}}}
\author{Genggeng Huang\footnote{Email address: genggenghuang@fudan.edu.cn.} and Yingshu L\"{u}\footnote{Corresponding author. Email address: yingshulv@fudan.edu.cn. }}
\date{}
\maketitle
\begin{center}
School of Mathematical Sciences, Fudan University, Shanghai, China
\end{center}
\date{}
\maketitle
\begin{abstract}
This paper is devoted to study the following degenerate Monge-Amp{\`e}re equation:
\begin{eqnarray}\label{ab1}
\begin{cases}
\det D^2 u=\Lambda_q  (-u)^q \quad \text{in}\quad \Omega,\\
u=0 \quad\text{on}\quad \partial\Omega
\end{cases}
\end{eqnarray}
for some positive constant $\Lambda_q$.
Suppose $\Omega\subset\subset \mathbb R^n$ is  uniformly convex and analytic. Then the solution of \eqref{ab1} is analytic in $\bar\Omega$ provided $q\in \mathbb Z^+$.
\end{abstract}



\section{Introduction}
In this paper, we focus on the analyticity of the solution of the following Monge-Amp{\`e}re equation:
\begin{eqnarray}\label{intro1}
\begin{cases}
\det D^2 u=\Lambda_q (-u)^q \quad \text{in}\quad \Omega,\\
u=0 \quad \text{on} \quad\partial \Omega,
\end{cases}
\end{eqnarray}
where $q>0$ and $\Omega$ is a bounded convex domain in $\mathbb R^n$.

This problem was first studied by Lions\cite{Lions1985}. In \cite{Lions1985}, Lions proved that for $q=n$, \eqref{intro1} admits a unique eigenvalue $\Lambda_n$ and eigenfunction $u\in C^{1,1}(\bar\Omega)\cap C^{\infty}(\Omega)$(up to multiplications of positive constants) provided $\Omega$ is smooth and uniformly convex.
\par Later, Chou \cite{Chou1990} studied the problem \eqref{intro1} for $q>0$. His approach is based on the Monge-Amp{\`e}re functional
\begin{eqnarray}\label{intro2}
J(u)=\frac{1}{n+1}\int_{\Omega}(-u)\det D^2 udx-\frac 1{q+1}\int_{\Omega}|u|^{q+1}dx
\end{eqnarray}
and the following logarithmic gradient flow
\begin{equation}
\left\{\begin{array}{ll}
\frac{\partial u}{\partial t}=\ln \det(D_x^2 u)-q\ln (-u),\quad (x,t)\in \Omega\times (0,+\infty),\\
u(x,0)=u_0(x),\quad x\in \Omega,\\
u=0 \quad \text{on}\quad \partial\Omega\times (0,+\infty).
\end{array}
\right.
\end{equation}

Then Chou obtained the existence of non-trivial solutions $u\in C^{0,1}(\bar\Omega)\cap C^\infty(\Omega)$ provided $\Omega$ is smooth and uniformly convex. Also the uniqueness of non-trivial solution  was established in \cite{Chou1990} for $0<q<n$ and $q=n$(up to multiplications of positive constants). Later, Hartenstine \cite{Hartenstine2009} extended it to bounded and strictly convex domain $\Omega$ for $0<q<n$. Recently, Le \cite{Le2017} proved the existence($q>0$) of solutions of \eqref{intro1} and the uniqueness($0<q\le n$) of solutions of \eqref{intro1}  for general bounded convex domain $\Omega$. For $q>n$, the uniqueness of non-trivial solutions of \eqref{intro1} is less known. Recently, the first named author \cite{Huang2019} proved that there exists a constant $\varepsilon(n)>0$ such that \eqref{intro1} admits a unique non-trivial solution for $q\in(n,n+\varepsilon(n))$.
\par The higher order global regularity of the solutions of \eqref{intro1} remains unknown until rencent 10 years. The main difficulty arises from the degeneracy of the equation \eqref{intro1} on the boundary $\partial\Omega$. In a survey paper, Trudinger and Wang [\cite{TrudingerWang2008},P21] proposed the problem that whether $u$ is smooth up to the boundary for $q=n$ when $\Omega$ is smooth and uniformly convex. Later, Hong, Huang and Wang \cite{HongHuangWang2011} gave an affirmative answer to this problem in dimension 2. Their approach relies on an auxiliary function
\begin{equation}
H=u_{22}u_{1}^2-2u_{12}u_1u_2+u_{11}u_2^2
\end{equation}
which is related to the Gauss curvature of the level set of $u$. For arbitrary dimensions, Savin \cite{Savin2014}
made a first contribution on the global $C^2$ regularity of solutions of \eqref{intro1}. Later, Le and Savin\cite{LeSavin2017} completely solved the above problem in arbitrary dimensions. The key observation of the work \cite{Savin2014} and \cite{LeSavin2017} is that near the boundary,
\begin{equation}
u(x)\sim \frac{1}{2}|x'|^2+\frac {1}{(q+1)(q+2)}x_n^{q+2}
\end{equation}
which allows them to use blow-up and perturbation arguments to show $u\in C^{2,\alpha}$.
Then they can raise the regularity up to $C^\infty$ by investigating a linear degenerate elliptic equation.
\par It is natural to ask that whether the solution of \eqref{intro1} is analytic up to the boundary provided the domain $\Omega$ is analytic and uniformly convex. The analyticity of the solutions of uniformly elliptic equations(systems) are well studied. We refer readers to \cite{MorreyNirenberg1957} for linear equations(systems), \cite{Morrey1958a,Morrey1958b} for non-linear equations(systems) and \cite{Friedman1958} for more general regularity results. There seems no unified results for degenerate elliptic cases. In the present case, the model degenerate equation is
\begin{eqnarray}\label{intro3}
\begin{cases}
\mathcal L(u)=u_{nn}+x_n^m\Delta_{x'} u=f \quad \text{in}\quad \mathbb R^n_+,\\
u(x',0)=g(x'),\quad x'\in \mathbb R^{n-1},
\end{cases}
\end{eqnarray}
where $x=(x_1,\cdots,x_{n-1},x_n)=(x',x_n)$. Usually, people call \eqref{intro3} a Grushin type degenerate elliptic equation. To the authors's best knowledge, there seems no results considering the analyticity of solutions of fully non-linear elliptic equations with degeneracy as in \eqref{intro3}.
\par Now we state our main results in the present paper.
\begin{theorem}\label{mainthm}
Suppose $u$ is a non-trivial solution of \eqref{intro1} and $\Omega\subset\subset \mathbb R^n$ is analytic and uniformly convex. Then $u$ is analytic in $\bar{\Omega}$ provided $q\in \mathbb Z^+$, i.e. $u\in C^\omega(\bar\Omega)$.
\end{theorem}
\begin{remark}
The idea of the proof of Theorem \ref{mainthm} originated from \cite{Kato1996} for semi-linear elliptic equation. For the fully non-linear uniform elliptic case, we refer the readers to \cite{Ha06}.
\end{remark}

The present paper is organized as follows. In Section 2, we will collect some basic estimates for the linear degenerate elliptic equation \eqref{intro3}. Then we apply the ideas of \cite{Kato1996} to show that $u$ is analytic up to the boundary in Section 3.

\section{Estimates for linear model equation}
In the present section, we consider the following linear degenerate elliptic equations:
\begin{eqnarray}\label{LM1}
\begin{cases}
u_{nn}+x_n^m\Delta_{x'} u=f \quad \text{in}\quad \mathbb R^n_+,\\
u(x',0)=g(x'),\quad x'\in \mathbb R^{n-1},
\end{cases}
\end{eqnarray}
where $x=(x_1,\cdots,x_{n-1},x_n)=(x',x_n)$.

Firstly, we introduce some notations. Let $m$ be a positive integer. For a multi-index $\alpha=(\alpha_1, \alpha_2,\cdots, \alpha_n)$, we denote $\partial_x^{\alpha}=\partial_1^{\alpha_1}\partial_2^{\alpha_2} \cdots \partial_n^{\alpha_n}$ with $|\alpha|=\alpha_1+\alpha_2+\cdots+\alpha_n$ and $u_i=\partial_iu=\frac{\partial u}{\partial {x_i}}$($i=1,2,\cdots,n$).

For $k\ge 0$, we define the following weighted Sobolev space $\widetilde W^{k,2}(G_1)$ as
\begin{equation*}
\begin{split}
\|u\|_{\widetilde W^{k,2}(G_1)}=&\sum_{|\alpha|=k}\|\partial_{n}^2\partial_x^{\alpha}u\|_{L^2(G_1)}+\sum_{|\alpha|=k+2,\alpha_n\le 1}\|x_n^m \partial_{x}^{\alpha} u\|_{L^2(G_1)}+\sum_{|\alpha|=k}\|\partial_n\partial^{\alpha}_x u\|_{L^2(G_1)}\\
+&\sum_{|\alpha|=k+1,\alpha_n=0}\|x_n^{m-1}\partial_{x}^{\alpha} u\|_{L^2(G_1)}+\|\partial_n u(x',0)\|_{H^{k}(\mathbb R^{n-1})}+\|u\|_{H^{k}(G_1)},
\end{split}
\end{equation*}
where $G_1=\mathbb R^{n-1}\times [0,1)$.

In the following, we introduce two lemmas which are the key ingredients to prove the analyticity of the solutions to Monge-Amp{\`e}re equations.
\begin{lemma}\label{lem1}
Suppose $u\in C_c^\infty(\overline{\mathbb R^n_+})$ solves \eqref{LM1} with supp $u\subset G_1$. Then there holds
\begin{eqnarray}\label{aa1}
\|u\|_{\widetilde W^{0,2}(G_1)}\le C(\|f\|_{L^2(G_1)}+\|g\|_{H^1(\mathbb R^{n-1})})
\end{eqnarray}
for some constant $C$ depending only on $n$.
\end{lemma}
\begin{proof}
We divide the proof of the present lemma into two cases.
\begin{itemize}
	\item[Case 1.]\par For $m=1$, we follow the steps of standard $W^{2,2}$-estimates as Laplacian equation. By \eqref{LM1}, one gets
\begin{eqnarray}\label{es1}
\int_{\mathbb R^n_+}(u_{nn}+x_n\Delta_{x'}u)^2dx=\int_{\mathbb R^n_+} f^2 dx.
\end{eqnarray}
Integrating by parts the crossproduct term yields
\begin{equation}
\begin{split}
& 2\int_{\mathbb R^n_+} x_n u_{nn} \Delta_{x'}u dx=-2\int_{\mathbb R^n_+} x_n \nabla_{x'}u_{nn} \cdot \nabla_{x'}u dx\\ =& 2\int_{\mathbb R^n_+} x_n |\nabla_{x'}u_{n}|^2 dx+2\int_{\mathbb R^n_+}\nabla_{x'}u_n\cdot\nabla_{x'} udx\\
=& 2\int_{\mathbb R^n_+} x_n |\nabla_{x'}u_{n}|^2 dx-\int_{\mathbb R^{n-1}}|\nabla_{x'}g|^2 dx'.
\end{split}
\end{equation}
Integrating by parts the other terms of \eqref{es1}, one gets
\begin{eqnarray}\label{es2}
\begin{split}
\int_{\mathbb R^n_+}u^2_{nn}+x_n^2|\nabla_{x'}^2u|^2+2\int_{\mathbb R^n_+} x_n |\nabla_{x'}u_{n}|^2 dx=\int_{\mathbb R^n_+}f^2 dx+\int_{\mathbb R^{n-1}}|\nabla_{x'}g|^2 dx'.
\end{split}
\end{eqnarray}
Multiplying $-u$ on both sides of \eqref{LM1} and integrating by parts, one obtains
\begin{eqnarray}\label{es3}
\begin{split}
\int_{\mathbb R^n_+}u_n^2+x_n|\nabla_{x'} u|^2 dx=-\int_{\mathbb R^{n-1}}u(x',0)u_n(x',0)dx'-\int_{\mathbb R^n_+}fudx.
\end{split}
\end{eqnarray}

\quad \quad Suppose
\begin{eqnarray}\label{es3.1}
\int_{\mathbb R^{n-1}}u^2_n(x',t)dx'=\inf_{x_n\in [0,1]}\int_{\mathbb R^{n-1}}u^2_n(x',x_n) dx'.
\end{eqnarray}
Then we know
\begin{equation}
\begin{split}
&\left|\int_{\mathbb R^{n-1}}u(x',0)u_n(x',0)dx'\right|\le C_\varepsilon\int_{\mathbb R^{n-1}}g^2dx'+\varepsilon\int_{\mathbb R^{n-1}}u_n^2(x',0)dx'\\
\le & C_\varepsilon\int_{\mathbb R^{n-1}}g^2dx'+C\varepsilon(\int_{\mathbb R^{n-1}}u_n^2(x',t)dx'+\int_{\mathbb R^n_+}u_{nn}^2 dx)\\
\le & C(\|g\|_{H^1(\mathbb R^{n-1})}^2+\|f\|_{L^2(\mathbb R^n_+)}^2)+\frac 14\int_{\mathbb R^{n}_+}u_n^2 dx.
\end{split}
\end{equation}
In getting the last inequality, we use \eqref{es2}, \eqref{es3.1} and take suitable $\varepsilon>0$ small.
Also, one has
\begin{equation}
\begin{split}
\int_{\mathbb R^n_+} u^2 dx\le C \int_{\mathbb R^{n-1}} g^2 dx'+\int_{\mathbb R^{n}_+}u_n^2 dx
\end{split}
\end{equation}
since $u$ has compact support.
Noticing that $\nabla_{x'}u=\partial_n(x_n \nabla_{x'}u)-x_n \nabla_{x'}u_{n}$, one knows
\begin{eqnarray}\label{es4}
\begin{split}
&\int_{\mathbb R^n_+}|\nabla_{x'} u|^2 dx\le 2\int_{\mathbb R^n_+}|\partial_n(x_n \nabla_{x'}u)|^2+x_n^2 |\nabla_{x'}u_{n}|^2 dx\\
= & 2\int_{\mathbb R^n_+} x_n^2 |\nabla_{x'}u_{n}|^2 dx-2\int_{\mathbb R^n_+} x_n\nabla_{x'}u\cdot (2\nabla_{x'}u_{n}+x_n\nabla_{x'}u_{nn})dx\\
= &4\int_{\mathbb R^n_+} x_n^2 |\nabla_{x'}u_{n}|^2dx.
\end{split}
\end{eqnarray}
Combining estimates \eqref{es2}-\eqref{es4} and the assumption supp $u\subset G_1$ yield the present lemma for $m=1$.
\item[Case 2.] For $m\ge 2$, the above method fails. We need employ some previous estimates for \eqref{LM1}.
One may view the solutions of \eqref{LM1} as $u=v+w$ where $v$ solves
\begin{eqnarray}\label{LM2}
\begin{cases}
\partial_{nn}v+x_n^m\Delta_{x'}v=0 \quad \text{in}\quad \mathbb R^n_+,\\
 v(x',0)=g(x'),\quad x'\in \mathbb R^{n-1},
\end{cases}
\end{eqnarray}
and $w$ solves
\begin{eqnarray}\label{LM3}
\begin{cases}
\partial_{nn}w+x_n^m\Delta_{x'}w=f \quad \text{in}\quad \mathbb R^n_+,\\
w(x',0)=0,\quad x'\in \mathbb R^{n-1}.
\end{cases}
\end{eqnarray}

\quad \quad Since $g(x')$ is compactly supported in $\mathbb R^{n-1}$, Hong-Wang [Lemma 3.1,\cite{HongWang2009}] proved that there exists an operator $\mathcal B$ such that $v(x)=\mathcal B(g)(x)$ solves \eqref{LM2} and satisfies the following estimates
\begin{eqnarray}\label{LM4}
\begin{split}
&\|\partial_{nn}(\mathcal B(g))(x_n)\|_{L^2_{x'}(\mathbb R^{n-1})}+\|x_n^{\frac m2}\Lambda_1\partial_n (\mathcal B(g))(x_n)\|_{L^2_{x'}(\mathbb R^{n-1})}\\&+\|x_n^m\Lambda_1^2(\mathcal B(g))(x_n)\|_{L^2_{x'}(\mathbb R^{n-1})}
+\|\partial_n(\mathcal B(g))(x',0)\|_{L^2_{x'}(\mathbb R^{n-1})}\\&+\|x_n^{\frac m2}\Lambda_1^{\frac{m+4}{m+2}}(\mathcal B(g))(x_n)\|_{L^2_{x'}(\mathbb R^{n-1})}\le C\|\Lambda^{\frac 4{m+2}}g\|_{L^2(\mathbb R^{n-1})},
\end{split}
\end{eqnarray}
where $C$ is a universal constant depending only on $n$, $\Lambda_1$ and $\Lambda$ represent the singular integral operators with symbols $|\xi|$ and $(1+|\xi|^2)^{\frac 12}$, $\xi=(\xi_1,\cdots,\xi_{n-1})$. Also the norm $L^2_{x'}(\mathbb R^{n-1})$ represents the $L^2-$norm over $\mathbb R^{n-1}$ in $x'$ variable.
Then by a simple integration over $x_n$, one obtains
\begin{eqnarray}\label{LM5}
\begin{split}
&\|\partial^2_{n}(\mathcal B(g))\|_{L^2(G_1)}+\|x_n^{\frac m2}\Lambda_1\partial_n (\mathcal B(g))\|_{L^2(G_1)}+\|x_n^m\Lambda_1^2(\mathcal B(g))\|_{L^2(G_1)}\\+&\|\partial_n(\mathcal B(g))\|_{L^2(G_1)}
+\|x_n^{\frac m2}\Lambda_1^{\frac{m+4}{m+2}}(\mathcal B(g))\|_{L^2(G_1)}+\|\partial_n(\mathcal B(g))(x',0)\|_{L^2(\mathbb R^{n-1})}\\+&\|\mathcal B(g)\|_{L^2(G_1)}\le C\|\Lambda^{\frac 4{m+2}}g\|_{L^2(\mathbb R^{n-1})}.
\end{split}
\end{eqnarray}

\quad \quad Since $f$ is compactly supported in $\overline{\mathbb R^n_+}$. Hong-Li [Theorem 3.2,\cite{HongLi1996}](also see Lemma 3.4 in \cite{HongWang2009}) proved that there exists an operator $\mathcal T$ such that $w=\mathcal T(f)$ solves \eqref{LM3} and satisfies the following estimates
\begin{eqnarray}\label{LM6}
\begin{split}
&\|\partial_{n}^2(\mathcal Tf)\|_{L^2(G_1)}+\|x_n^m \Lambda\partial_n(\mathcal Tf)\|_{L^2(G_1)}+\|x_n^m\Lambda^2(\mathcal Tf)\|_{L^2(G_1)}+\|\partial_n\mathcal Tf\|_{L^2(G_1)}\\
+&\|x_n^{\frac m2}\Lambda (\mathcal Tf)\|_{L^2(G_1)}+\|\partial_n\mathcal Tf(x',0)\|_{L^2(\mathbb R^{n-1})}+\|\mathcal Tf\|_{L^2(G_1)}\le C \|f\|_{L^2(G_1)}.
\end{split}
\end{eqnarray}
Then combining \eqref{LM5} and \eqref{LM6}, one gets
\begin{eqnarray}\label{LM7}
\begin{split}
&\|\partial_{n}^2u\|_{L^2(G_1)}+\sum_{i=1}^{n-1}\|x_n^m \partial_{in} u\|_{L^2(G_1)}+\sum_{i,j=1}^{n-1}\|x_n^m\partial_{ij} u\|_{L^2(G_1)}+\|\partial_n u\|_{L^2(G_1)}\\
+&\sum_{i=1}^{n-1}\|x_n^{\frac m2}\partial_{i} u\|_{L^2(G_1)}+\|\partial_n u(x',0)\|_{L^2(\mathbb R^{n-1})}+\|u\|_{L^2(G_1)}\\
\le& C \left(\|f\|_{L^2(G_1)}+\|g\|_{H^1(\mathbb R^{n-1})}\right).
\end{split}
\end{eqnarray}
The last inequality of \eqref{LM7} comes from
\begin{equation*}
\|\Lambda^{\frac 4{m+2}} g\|_{L^2(\mathbb R^{n-1})}=\|(1+|\xi|^2)^{\frac 2{m+2}}\hat g\|_{L^2(\mathbb R^{n-1})}\le \|(1+|\xi|^2)^{\frac{1}{2}} \hat g\|_{L^2(\mathbb R^{n-1})}\le \|g\|_{H^1(\mathbb R^{n-1})}
\end{equation*}
because of $m\ge 2$.
Noticing that $\frac m2\le m-1$ for $m\ge 2$, one knows
\begin{equation*}
\|x_n^{\frac m2}\partial_{x'} u\|_{L^2(G_1)}\ge \|x_n^{m-1}\partial_{x'}u\|_{L^2(G_1)}
\end{equation*}
which implies the present lemma for $m\ge 2$.
	\end{itemize}
\end{proof}
In fact, differentiating the linear equation \eqref{LM1}, one can get the following lemma.
\begin{lemma}\label{lemest1}
Suppose $u\in C_c^\infty(\overline{\mathbb R^n_+})$ solves \eqref{LM1} with supp $u\subset G_1$. Then there holds
\begin{eqnarray}\label{lemest1.1}
\|u\|_{\widetilde W^{k,2}(G_1)}\le  C_k \left(\|f\|_{H^{k}(G_1)}+\|g\|_{H^{k+1}(\mathbb R^{n-1})}\right),\quad k\ge 0
\end{eqnarray}
for some positive constant $C_k$ depending only on $k,m,n$.
\end{lemma}
\begin{proof}
For $k=0$, this is just \eqref{aa1}.  We now prove the present lemma by induction on $k$. Suppose the estimate \eqref{lemest1.1} is true for $k$, we need to prove \eqref{lemest1.1} holds for $k+1$.
\par
First, we consider $\partial_x^{\alpha}u=\partial_{x'}^{\alpha}u$, $|\alpha|=k+1$, $\alpha_n=0$. Differentiating \eqref{LM1} with respect to $x'$ for $k+1$ times, one gets
\begin{eqnarray}\label{LM8}
\begin{cases}
\partial^2_{n}\partial_{x'}^{\alpha}u +x_n^m\Delta_{x'} \partial_{x'}^{\alpha}u=\partial_{x'}^{\alpha}f \quad \text{in}\quad \mathbb R^n_+,\\
\partial_{x'}^{\alpha}u(x',0)=\partial_{x'}^{\alpha}g(x'),\quad x'\in \mathbb R^{n-1}.
\end{cases}
\end{eqnarray}
Applying \eqref{aa1} to \eqref{LM8}, one gets
\begin{eqnarray}\label{LM14}
\|\partial_{x'}^{\alpha}u\|_{\widetilde W^{0,2}(G_1)}\le C(\|\partial_{x'}^{\alpha}f\|_{L^2(G_1)}+\|\partial_{x'}^{\alpha}g\|_{H^1(\mathbb R^{n-1})}).
\end{eqnarray}
Especially, we obtain
\begin{eqnarray}\label{LM10}
\begin{split}
&\sum_{|\beta|=k+1,\beta_n=0}\|\partial_{n}\partial_{x'}^{\beta}u(x',0)\|_{L^2(\mathbb R^{n-1})}+\sum_{|\beta|=k+2,\beta_n=0}\|x_n^{m-1}\partial_{x'}^{\beta} u\|_{L^2(G_1)}\\
& \le C(\|f\|_{H^{k+1}(G_1)}+\|g\|_{H^{k+2}(\mathbb R^{n-1})}).
\end{split}
\end{eqnarray}
\par Then we consider the case $\partial_{x}^{\alpha}u=\partial_n\partial_{x'}^{\alpha'} u$ with $|\alpha'|=k$.  Differentiating \eqref{LM1} with respect to $x'$ for $k$ times and then with respect to $x_n$, one can derive
\begin{eqnarray}\label{LM9}
\partial_n^2\partial_n\partial_{x'}^{\alpha'}u+x_n^m\Delta_{x'} \partial_n\partial_{x'}^{\alpha'}u=\partial_n\partial_{x'}^{\alpha'}f-mx_n^{m-1}\Delta_{x'}\partial_{x'}^{\alpha'}u \quad \text{in}\quad \mathbb R^n_+.
\end{eqnarray}
For the boundary term, by \eqref{LM10}, we know
\begin{eqnarray}\label{LM10.1}
\|\partial^{\alpha'}_{x'}u_n(x',0)\|_{H^1({\mathbb R^{n-1}})}\leq C(\|f\|_{H^{k+1}(G_1)}+\|g\|_{H^{k+2}(\mathbb R^{n-1})}).
\end{eqnarray}
Applying \eqref{aa1} to \eqref{LM9} and using \eqref{LM10} and \eqref{LM10.1}, one gets
\begin{equation*}
\begin{split}
\|\partial_n\partial_{x'}^{\alpha'}u\|_{\widetilde W^{0,2}(G_1)}\le& C(\|f\|_{H^{k+1}(G_1)}+\|x_n^{m-1}\Delta_{x'}\partial_{x'}^{\alpha'}u\|_{L^2(G_1)}+\|\partial^{\alpha'}_{x'}u_n(x',0)\|_{H^1(\mathbb R^{n-1})})\\
\le &C(\|f\|_{H^{k+1}(G_1)}+\|g\|_{H^{k+2}(\mathbb R^{n-1})}).
\end{split}
\end{equation*}
\par The last case is $\partial^{\alpha}_x u=\partial_n^l \partial^{\beta}_{x'} u$, $l\ge 2$, $|\beta|=k+1-l$. Differentiating \eqref{LM1} with respect to $x'$ for $k+1-l$ times and then with respect to $x_n$ for $l$ times, one can derive
\begin{eqnarray}\label{LM13}
\begin{split}
&\partial_n^2(\partial_n^l \partial^{\beta}_{x'} u)+x_n^m\Delta_{x'} (\partial_n^l \partial^{\beta}_{x'} u)\\
=&\partial_n^l \partial^{\beta}_{x'} f-\sum_{r=1}^{\min(m,l)}\frac{l!}{r!(l-r)!}\partial_n^r(x_n^m)\partial_n^{l-r}\Delta_{x'}\partial_{x'}^\beta u \quad \text{in}\quad \mathbb R^n_+.
\end{split}
\end{eqnarray}

We first consider the boundary term. Applying  $\partial_n^{l-2}\partial_{x'}^\beta$ to \eqref{LM1}, one gets
\begin{eqnarray}\label{LM12}
\partial_n^l \partial^{\beta}_{x'} u=-\partial_{n}^{l-2}\partial_{x'}^{\beta}(x_n^m\Delta_{x'}u)+\partial_n^{l-2}\partial_{x'}^\beta f.
\end{eqnarray}
 Restricting  \eqref{LM12} on $x_n=0$, then we obtain
\begin{eqnarray}\label{1}
\|\partial_n^{l-2}\partial_{x'}^\beta f(x',0)\|_{H^1(\mathbb R^{n-1})}\le C\|f\|_{H^{k+1}(G_1)}
\end{eqnarray}
by the trace theorem of Sobolev space. For the term $\partial_{n}^{l-2}\partial_{x'}^{\beta}(x_n^m\Delta_{x'}u)$,
it matters only if $l\ge m+2$ and equals to $ m! C_{l-2}^m\partial_n^{l-2-m}\partial_{x'}^\beta \Delta_{x'}u(x',0)$.
\par If  $l=m+2$, then we know
\begin{equation*}
\|\partial_{x'}^\beta \Delta_{x'}u(x',0)\|_{H^1(\mathbb R^{n-1})}\le \|g\|_{H^{k+1}(\mathbb R^{n-1})}.
\end{equation*}
\par If $l\ge m+3$, then we know $\partial_n^{l-2-m}\partial_{x'}^\beta \Delta_{x'}u(x',0)=\partial_n(\partial_n^{l-3-m}\partial_{x'}^\beta \Delta_{x'}u)(x',0)$. This implies
\begin{eqnarray}\label{2}
\|\partial_n^{l-2-m}\partial_{x'}^\beta \Delta_{x'}u(x',0)\|_{H^1(\mathbb R^{n-1})}&\leq & C\|\partial_n\partial_x^{\tilde \beta}u(x',0)\|_{H^1(\mathbb R^{n-1})}\\ \nonumber
&\leq& C_k(\|f\|_{H^k(G_1)}+\|g\|_{H^{k+1}(\mathbb R^{n-1})})
\end{eqnarray}
by the induction assumption for some $\tilde \beta$ with $|{\tilde \beta}|=k-m\le k-1$.

Combining \eqref{1}-\eqref{2}, one obtains
\begin{eqnarray}\label{3}
\|\partial_n^l \partial^{\beta}_{x'} u(x',0)\|_{H^1(\mathbb R^{n-1})}\le C_k(\|f\|_{H^k(G_1)}+\|g\|_{H^{k+1}(\mathbb R^{n-1})}).
\end{eqnarray}

According to \eqref{aa1}, the remaining thing is to analysis the term $\displaystyle\sum_{r=1}^{\min(m,l)}\partial_n^r(x_n^m)\partial_n^{l-r}\Delta_{x'}\partial_{x'}^\beta u$.  We only need to take care of $x_n^{m-1}\partial_n^{l-1}\partial_{x'}^{\beta}\Delta_{x'}u$, $x_n^{m-2}\partial_n^{l-2}\partial_{x'}^{\beta}\Delta_{x'}u(m\ge 2)$, $l\ge 2$, $|\beta|=k+1-l$.

For the term $x_n^{m-1}\partial_n^{l-1}\partial_{x'}^{\beta}\Delta_{x'}u$:
\begin{itemize}
\item[(1)]If  $l=2$, i.e.
$x_n^{m-1}\partial_n^{l-1}\partial_{x'}^{\beta}\Delta_{x'}u=x_n^{m-1}\partial_n\partial_{x'}^{\beta}\Delta_{x'}u$, $|\beta|=k-1$. Then by estimate \eqref{LM14}, this term can be controlled by $\|f\|_{H^{k+1}(G_1)}+\|g\|_{H^{k+2}(\mathbb R^{n-1})}$.

\item[(2)] If $l\ge 3$, then by induction assumption on $k$, we have
\begin{equation*}
\begin{split}
&\|x_n^{m-1}\partial_n^{l-1}\partial_{x'}^{\beta}\Delta_{x'}u\|_{L^2(G_1)}\le \|\partial_n^2(\partial_n^{l-3}\partial_{x'}^\beta \Delta_{x'}u)\|_{L^2(G_1)}\\
 \le& \|u\|_{\widetilde W^{k,2}(G_1)}\le C_k(\|f\|_{H^{k}(G_1)}+\|g\|_{H^{k+1}(\mathbb R^{n-1})}).
 \end{split}
\end{equation*}
\end{itemize}

Similarly, for the term $x_n^{m-2}\partial_n^{l-2}\partial_{x'}^{\beta}\Delta_{x'}u(m\ge 2)$:
\begin{itemize}
\item[(1)]If $l=2$, i.e. $x_n^{m-2}\partial_n^{l-2}\partial_{x'}^{\beta}\Delta_{x'}u=x_n^{m-2}\partial_{x'}^{\beta}\Delta_{x'}u$, $|\beta|=k-1$. Then by estimate \eqref{LM14}, this term can be controlled by $\|f\|_{H^{k+1}(G_1)}+\|g\|_{H^{k+2}(\mathbb R^{n-1})}$.

\item[(2)]If $l\ge 3$, then by induction assumption on $k$, we have
\begin{equation*}
\begin{split}
&\|x_n^{m-2}\partial_n^{l-2}\partial_{x'}^{\beta}\Delta_{x'}u\|_{L^2(G_1)}\le \|\partial_n(\partial_n^{l-3}\partial_{x'}^\beta\Delta_{x'}u)\|_{L^2(G_1)} \\
\le &\|u\|_{\widetilde W^{k,2}(G_1)}\le C_k(\|f\|_{H^{k}(G_1)}+\|g\|_{H^{k+1}(\mathbb R^{n-1})}).
\end{split}
\end{equation*}
\end{itemize}

Overall, we obtain
\begin{equation*}
\|\partial_n^l \partial^{\beta}_{x'} u\|_{\widetilde W^{0,2}(G_1)}\le C(\|f\|_{H^{k+1}(G_1)}+\|g\|_{H^{k+2}(\mathbb R^{n-1})}),\quad \forall l\ge 2, |\beta|=k+1-l.
\end{equation*}
This ends the proof of present lemma.
\end{proof}

\begin{lemma}\label{lem2.3}
For any $u,v\in \widetilde W^{k,2}(G_1)$,   there exists a constant $C_k$ such that
\begin{equation}
\|uv\|_{\widetilde W^{k,2}(G_1)}\le C_k \|u\|_{\widetilde W^{k,2}(G_1)}\|v\|_{\widetilde W^{k,2}(G_1)}
\end{equation}
provided $k\ge n+3$.
\end{lemma}
\begin{proof}
For any multi-index $\alpha\in \mathbb N^n$, one has
\begin{equation*}
\partial_x^\alpha(uv)=\sum_{\beta+\gamma=\alpha}\frac{\alpha!}{\beta!\gamma!}\partial_x^{\beta}u\partial_x^{\gamma}v.
\end{equation*}
Since the highest order of derivative in $\widetilde W^{k,2}(G_1)$ is $k+2$, we know
$\min(|\beta|,|\gamma|)\le \frac {k+2}2$.  Then for the product $\partial_x^{\beta}u \partial_x^{\gamma}v$, at least one term is in $H^{k-\frac {k+2}2}(\mathbb R^n_+)\hookrightarrow L^\infty(\mathbb R^n_+)$ provided $k>n+2$ by Sobolev embedding theorem. This implies the present lemma.
\end{proof}
In the following, we give a lemma which is essentially Lemma 1 in \cite{Friedman1958}.
\begin{lemma}\label{lem1}
Let $B_1\times \mathbb B_R$ be the domain in $\mathbb R^n\times \mathbb R^L$. Assume that
$\Phi(x,y)$ is a polynomial and $p$ is a positive integer. Let $\eta\in C_c^\infty(B_1)$ is a cut-off function.
 Then, there exist positive constants $A_0$, $\widetilde A_0$ and $A_1$, depending only on $n$,  $L$, $k$, $\eta$ and the polynomial $\Phi(x,y)$, such that, for any $C^p$-function $y=(y_1,\cdots,y_L): B_1\rightarrow \mathbb B_R$, if for any $x\in B_1$ and any non-negative integer $l\le p$,
 \begin{equation*}
 \sum_{i=1}^L \|\eta^l \partial_x^l y_i(x)\|_{\widetilde W^{k,2}(B_1^+)}\le A_0A_1^{(l-2)^+}{(l-2)^+}!,
 \end{equation*}
 for some $k>n+2$.
Then, for any $x\in B_1^+$,
\begin{equation*}
\|\eta^p \partial_x^p[\Phi(x,y(x))]\|_{\widetilde W^{k,2}(B_1^+)}\le \widetilde A_0 A_1^{(p-2)^+}(p-2)^+!.
\end{equation*}
Here, if no confusion occurs, the meaning of $l$ and $p$ can be vary from multi-index to pure integer.
\end{lemma}
\begin{proof}
By our assumptions, we know
\begin{equation*}
\Phi(x,y(x))=\sum_{|\alpha|+|\beta|\le d} C_{\alpha\beta} x^\alpha y(x)^\beta.
\end{equation*}
Here $d$ is the degree of the polynomial $\Phi(x,y)$ and $x^\alpha=x_1^{\alpha_1}\cdots x_n^{\alpha_n}$.
Then for $x^\alpha y(x)^\beta=x^\alpha y_{i_1}(x)\cdots y_{i_{|\beta|}}(x)$, $1 \leq i_1,\cdots, i_{|\beta|} \leq L $, one knows
\begin{eqnarray}\label{701}
\partial_x^{p} (x^\alpha y(x)^\beta)=\sum_{k_0+k_1+\cdots+k_{|\beta|}=p} \frac{p!}{k_0!k_1!\cdots k_{|\beta|}!} \partial_x^{k_0}(x^\alpha)\partial_{x}^{k_1}(y_{i_1}(x))\cdots \partial_x^{k_{|\beta|}}(y_{i_{|\beta|}}(x)).
\end{eqnarray}
Let  $C_0,C_1$ be the positive constants such that
\begin{equation*}
\begin{split}
\sum_{|\alpha|\le d}\|\eta^{l}\partial_x^{l}(x^\alpha)\|_{\widetilde W^{k,2}(B_1^+(0))}\le C_0C_1^{(l-2)^+}(l-2)^+!,\quad l\ge 0.
\end{split}
\end{equation*}
Let $\|\cdot\|$ be $\|\cdot\|_{\widetilde W^{k,2}(B_1^+(0))}$.  Then by Lemma \ref{lem2.3} and the assumptions, we know
\begin{eqnarray}
\begin{split}
&\|\eta^p\partial_x^{p} (x^\alpha y(x)^\beta)\|\\
\le& C_k^{|\beta|+1}\sum_{k_0+k_1+\cdots+k_{|\beta|}=p} \frac{p!}{k_0!k_1!\cdots k_{|\beta|}!} C_0C_1^{(k_0-2)^+}(k_0-2)^+!\Pi_{j=1}^{|\beta|}A_0A_1^{(k_j-2)^+}(k_j-2)!\\
\le & (A_0C_k)^{|\beta|+1}A_1^{(p-2)^+}p!\sum_{k_0+k_1+\cdots+k_{|\beta|}=p}\frac{(k_0-2)^+!(k_1-2)^+!\cdots (k_{|\beta|}-2)^+!}{k_0!k_1!\cdots k_{|\beta|}!}
\end{split}
\end{eqnarray}
provided $A_0\ge C_0,A_1\ge C_1$.
In the following, we will  show that  there exists a constant $C_2>0$  such that
\begin{eqnarray}\label{cl1}
\sum_{k_0+k_1+\cdots+k_{|\beta|}=p}\frac{(k_0-2)^+!(k_1-2)^+!\cdots (k_{|\beta|}-2)^+!}{k_0!k_1!\cdots k_{|\beta|}!}\le \frac{C_2(8\pi^2(d+1))^{|\beta|+1}}{(p+1)^2},\quad \forall p\in \mathbb N
\end{eqnarray}
by induction for all $|\beta|\le d$. It is easy to see that, we can choose $C_2$ large enough such that \eqref{cl1} holds for all $0\le p\le 10$. Suppose \eqref{cl1} holds for $p\ge 10$. Then for $p+1$, one has
\begin{equation*}
\begin{split}
&\sum_{k_0+k_1+\cdots+k_{|\beta|}=p+1}\frac{(k_0-2)^+!(k_1-2)^+!\cdots (k_{|\beta|}-2)^+!}{k_0!k_1!\cdots k_{|\beta|}!}\\
\le &(d+1)\sum_{k_{|\beta|}=1}^{p+1}\frac{(k_{|\beta|}-2)^+!}{k_{|\beta|}!} \sum_{k_0+\cdots+k_{|\beta|-1}=p+1-k_{|\beta|}}\frac{(k_0-2)^+!(k_1-2)^+!\cdots (k_{|\beta|-1}-2)^+!}{k_0!k_1!\cdots k_{|\beta|-1}!}\\
\le &(d+1)C_2(8\pi^2(d+1))^{|\beta|}\sum_{k_{|\beta|}=1}^{p+1}\frac{(k_{|\beta|}-2)^+!}{k_{|\beta|}!(p+2-k_{|\beta|})^2}\\
\le & 4(d+1)C_2(8\pi^2(d+1))^{|\beta|}\sum_{k_{|\beta|}=1}^{p+1}\frac{1}{k_{|\beta|}^2(p+2-k_{|\beta|})^2}\\
\le &4(d+1)C_2(8\pi^2(d+1))^{|\beta|}\left(\frac{9}{4(p+1)^2}\left(\sum_{k_{|\beta|}=1}^{[\frac{p+1}3]}\frac{1}{k_{|\beta|}^2}+\sum_{k_{|\beta|}=2[\frac{p+1}3]}^{p+1}\frac{1}{(p+2-k_{|\beta|})^2}\right)\right.\\
+&\left.\sum_{k_{|\beta|}=[\frac{p+1}3]}^{2[\frac{p+1}3]}\frac{1}{k_{|\beta|}^2(p+2-k_{|\beta|})^2}\right)
\le\frac{C_2(8\pi^2(d+1))^{|\beta|+1}}{(p+2)^2}.
\end{split}
\end{equation*}
In getting the last inequality of the above, we used
\begin{equation*}
\sum_{l=1}^{+\infty}\frac{1}{l^2}=\frac{\pi^2}{6},\quad \frac{27}{p+1}\le \pi^2,\quad \frac{(p+2)^2}{(p+1)^2}\le 2.
\end{equation*}
This ends the proof of \eqref{cl1}. By \eqref{cl1}, one knows
\begin{equation*}
\begin{split}
\|\eta^p\partial_x^{p} (x^\alpha y(x)^\beta)\| \le (8\pi^2 A_0C_k(d+1))^{d+1}A_1^{(p-2)^+}(p-2)^+!.
\end{split}
\end{equation*}
This implies
\begin{equation*}
\begin{split}
\|\eta^p \partial_x^p[\Phi(x,y(x))]\| \le (8\pi^2 A_0C_k(d+1))^{d+1}A_1^{(p-2)^+}(p-2)^+!\sum_{|\alpha|+|\beta|\le d} |C_{\alpha\beta}|.
\end{split}
\end{equation*}
By letting $\tilde A_0=(8\pi^2 A_0C_k(d+1))^{d+1}\sum_{|\alpha|+|\beta|\le d} |C_{\alpha\beta}|$, we finish the proof of present lemma.
\end{proof}

\section{Analyticity of the solutions}
 Through out this section,  we let $q=m\in \mathbb Z^+$ in \eqref{intro1}. Then   $u$ solves
\begin{eqnarray}\label{301}
\begin{cases}
\det D^2 u=\Lambda_m(-u)^m \quad \text{in}\quad \Omega,\\
u=0 \quad \text{on} \quad\partial \Omega.
\end{cases}
\end{eqnarray}
In the following, we omit the constant $\Lambda_m$. Since \eqref{301} is uniformly elliptic in the interior of $\Omega$, then applying the results of \cite{Friedman1958,Morrey1958a}, we can get $u\in C^\omega(\Omega)$. The remaining thing is to show $u$ is analytic up to the boundary.
\par  Suppose $\Omega\subset\mathbb R^n_+$, $0\in \partial\Omega$ and $x_n=0$ is the supporting plane of $\Omega$ at $0$. Then we know the non-trivial solution $u$ of \eqref{301} satisfies
\begin{itemize}
\item[(I)] $u\in C^\infty(\overline{\Omega})$ and $u_n(0)<0$, $\nabla_{x'}u(0)=0$.
\item[(II)] $\lambda I\le D^2_{x'}u(0)\le \Lambda I$ for two positive constants $\lambda,\Lambda$.
\end{itemize}
For the proof of (I) and (II), we refer the readers  to [Theorem B, \cite{HongHuangWang2011}] for $n=2$ and [Theorem 1.4, \cite{LeSavin2017}] for arbitrary dimension.

 Then as in \cite{LeSavin2017}, we take the following Hodo-graph transformation near $0$,
 \begin{equation*}
y_n=-x_{n+1},\quad y_{n+1}=x_n, \quad y_k=x_k (1\le k\le n-1).
\end{equation*}
In the new coordinates, the graph of $u$ near $0$ can be represented as
\begin{equation*}
y_{n+1}=v(y) \quad\text{in}\quad  \{y\in\mathbb R^n: \quad y_n>0\}\cap B_\delta(0)
\end{equation*}
for some $\delta>0$ small enough.
This implies
\begin{eqnarray}\label{333}
u(y',v(y))+y_n=0 \quad \text{in}\quad \{y\in\mathbb R^n: \quad y_n>0\}\cap B_\delta(0).
\end{eqnarray}
Then,  differentiating  the above equation directly yields
\begin{equation*}
u_\alpha+u_nv_\alpha=0,\quad u_nv_n+1=0,\quad \alpha=1,\cdots,n-1.
\end{equation*}
Since the Gauss curvature will not change, one has
\begin{eqnarray}\label{hodograph}
\begin{cases}
\det D^2v=K(1+|\nabla v|^2)^{\frac{n+2}2}=\det D^2 u\frac{(1+|\nabla v|^2)^{\frac{n+2}2}}{(1+|\nabla u|^2)^{\frac{n+2}2}}=y_n^m v_n^{n+2} \quad \text{in}\quad B^+_{\delta}(0),\\
v(y',0)=\phi(y'),\quad |y'|\le \delta,
\end{cases}
\end{eqnarray}
where $y_n=\phi(y')$ represents the boundary $\partial\Omega$ near $0$.
Moreover, the Hodo-graph transformation preserves the analyticity.
\begin{lemma}\label{lemhodo}
Let $v$ be given by \eqref{333} where $u$ is the non-trivial solution of \eqref{301}.
Then $v$ is analytic near $0$ implies $u$ is analytic near $0$.
\end{lemma}
\begin{proof}
Let $F(y,y_{n+1})=v(y)-y_{n+1}$ and $F(0)=0$. Then by the assumption of the present lemma, we know $F(y,y_{n+1})$ is analytic in $B_\delta(0)$ for some $\delta>0$ small.
Also by (I), one has
\begin{equation*}
F_{n}(0,0)=v_n(0)=-\frac{1}{u_n(0)}\ne 0.
\end{equation*}
  By the real analytic implicit function theorem(Theorem 1.8.3,\cite{KrantzParks2002}), we know there exists $u$ analytic near $0$ such that
  \begin{equation*}
  F(y_1,\cdots,y_{n-1},u(y_1,\cdots,y_{n-1},y_{n+1}),y_{n+1})=0,
  \end{equation*}
  which implies the present lemma.
\end{proof}
Differentiating \eqref{333} with respect to $y'$ for two times, one obtains
\begin{equation*}
u_{\alpha\beta}+u_{\alpha n}v_{\beta}+u_{n\beta}v_{\alpha}+u_{nn}v_\alpha v_{\beta}+u_n v_{\alpha\beta}=0,\quad 1\le \alpha,\beta\le n-1.
\end{equation*}
This implies
\begin{equation}
v_{\alpha\beta}(0)=-\frac{u_{\alpha\beta}(0)}{u_n(0)}
\end{equation}
is a positive definite matrix.
Then we can perform the partial Legendre transformation to the solutions of \eqref{hodograph}:
\begin{equation}
z_i=v_i(y)\quad (i\le n-1),\quad z_n=y_n,\quad v^*(z)=y'\cdot\nabla_{y'}v-v(y).
\end{equation}
Set the partial Legendre transformation by $z=T(y)$.
Then $v^*$ should satisfy
\begin{eqnarray}\label{pl1}
\left\{\begin{array}{ll}
z_n^m(-v_n^*)^{n+2}\det D^2_{z'}v^*+v^*_{nn}=0 \quad \text{in}\quad \Sigma^+=T(B^+_\delta(0))\subset\mathbb R^n_+,\\
v^*=\phi^* \quad \text{on}\quad \partial\Sigma^+\cap \{z_n=0\},
\end{array}
\right.
\end{eqnarray}
where $\phi^*$ is the Legendre transformation of $\phi$. Also the partial Legendre transformation preserves the analyticity.
\begin{lemma}\label{lemplt}Let $v^*$ be the partial Legendre transformation of $v$.
Then $v^*$ is analytic near $0$ implies $v$ is analytic near $0$.
\end{lemma}
\begin{proof}
Consider
\begin{equation}
\begin{split}
& F_0(y',z,t)=v^*(z)-y'\cdot z'+t,\quad F_i(y',z,t)=y_i-\partial_{z_i}v^*,\quad i=1,\cdots,n-1,
\end{split}
\end{equation}
where $y'\in \mathbb R^{n-1},z\in\mathbb R^n, t\in \mathbb R$ and $z'=(z_1,\cdots,z_{n-1})$. Then $F_0(0)=F_i(0)=0$, $i=1,\cdots,n-1$.
By assumptions, we know  $F_0(y',z,t)$, $F_i(y',z,t)$ are real analytic functions and
\begin{equation}
\det\left(\frac{\partial (F_0,\cdots,F_{n-1})}{\partial(t,z')}\right)(0)=(-1)^{n-1}\det D_{z'}^2v^*(0)\ne 0.
\end{equation}
Again by  the real analytic implicit function theorem(Theorem 1.8.3,\cite{KrantzParks2002}), we know $F_0(\tilde y',\tilde z,\tilde t)=0$, $F_i(\tilde y',\tilde z,\tilde t)=0$ determine real analytic functions $\tilde t(y',z_n)$, $\tilde z'(y',z_n)$. This implies the present lemma.
\end{proof}
Hence, by the conclusions of Lemma \ref{lemhodo} and Lemma \ref{lemplt},  we only need to consider the following type of equation:
\begin{eqnarray}\label{MA-1}
\begin{cases}
x_n^m(-u_n)^{n+2}\det D^2_{x'}u+u_{nn}=0 \quad \text{in}\quad B_1^+,\\
u=\varphi \quad \text{on}\quad \{x_n=0\}\cap B_1(=B'_{1}).
\end{cases}
\end{eqnarray}
\begin{theorem}\label{mainthm2}
Suppose $u\in C^\infty(\overline{B_1^+})$  and $\varphi\in C^\omega(B'_{1})$  solve \eqref{MA-1}. Moreover, $u$ satisfies
\begin{equation*}
\lambda_{\min}(D^2_{x'}u)\ge c_0>0,\quad |u_n(x',0)|\ge c_0>0 \quad \text{in} \quad \overline{B_1^+}.
\end{equation*}
Then $u\in C^\omega(\overline{B^+_{\frac{1}{2}}})$.
\end{theorem}
\begin{remark}
Combining Theorem \ref{mainthm2}, Lemma \ref{lemhodo} and Lemma \ref{lemplt}, we know Theorem \ref{mainthm} holds.
\end{remark}
Since analyticity is a local property, we restrict our discussion on $B_{2r}^+$ for some $r$ small enough to be determined later.
Then our aim of the remaining paragraph is to show that there exist two positive constants $A_0,A_1$ such that
\begin{eqnarray}\label{ana1}
\|\eta^{N-2}\partial_x^{N}u\|_{\widetilde W^{k-i,2}} \le A_0 A_1^{(N-4)^+}(N-4-i)^+!,\quad i=0,1,2.
\end{eqnarray}
Here, we fix a large enough integer $k$ such that Lemma \ref{lem2.3} holds. $\eta$ is a cut-off function and has the following form
\begin{equation*}
\eta(x)=\chi(x_1)\cdots\chi(x_n),
\end{equation*}
where  $0\le \chi \le 1$ is a cut-off function satisfying
\begin{equation}
\chi(t)\equiv 1,\quad \text{in}\quad [-r,r],\quad \chi\equiv 0,\quad \text{in}\quad [-2r,2r]^c.
\end{equation}
In fact, we just need to consider $\eta$ in $\{x\in B_{2r}| x_n>0\}$.
And if no confusion occurs, the meaning of $N$ can vary from multi-index to pure positive integer.

In the following, we will prove \eqref{ana1} via induction. Suppose \eqref{ana1} is true for $0,1,2,\cdots,N$. We need to show it holds for $N+1$.

Firstly, differentiating  equation \eqref{MA-1} with respect to $x_l$, $l=1,\cdots,n-1$, one has
\begin{eqnarray}\label{6}
\begin{cases}
\displaystyle \sum_{i,j=1}^{n-1}x_n^m(-u_n)^{n+2}U^{ij}\partial_{ij}(\partial_l u)+\partial_{nn}\partial_l u=-\partial_l(x_n^m(-u_n)^{n+2})\det D^2_{x'}u \quad \text{in}\quad B_1^+,\\
\partial_lu=\partial_l\varphi \quad \text{on}\quad \{x_n=0\}\cap B_1,
\end{cases}
\end{eqnarray}
where $U^{ij}$ is the cofactor matrix of $D_{x'}^2 u$.
Set
$$G=-\partial_l(x_n^m(-u_n)^{n+2})\det D^2_{x'}u.$$
Then we know $G$ is a polynomial with arguments $x,u,\nabla u,\nabla^2 u$.
Without loss of generality(after a transformation of coordinates), we may assume
\begin{equation*}
(-u_n)^{n+2}U^{ij}(0)=\delta_{ij}.
\end{equation*}
Then we can rewrite the equation of \eqref{6} as
\begin{eqnarray}\label{302}
x_n^m\Delta_{x'}(\partial_l u)+\partial_{nn}\partial_lu=G+ \sum_{i,j=1}^{n-1}a^{ij}x_n^m\partial_{ij}\partial_lu \quad \text{in}\quad B_1^+,
\end{eqnarray}
where
\begin{equation*}
a^{ij}=\delta_{ij}-(-u_n)^{n+2}U^{ij},\quad a^{ij}(x)=O(|x|),\quad |x|<<1.
\end{equation*}

Differentiating \eqref{302} for $N$ times with respect to $x'$, one gets
\begin{eqnarray}\label{MA-2}
x_n^m\Delta_{x'}\partial_{x'}^{N+1}u+\partial_{nn}(\partial_{x'}^{N+1}u)=\partial_{x'}^{N}G+ \sum_{i,j=1}^{n-1}\partial_{x'}^N(a^{ij}x_n^m\partial_{ij}\partial_lu) \quad \text{in}\quad B_1^+.
\end{eqnarray}
Multiplying \eqref{MA-2} with $\eta^{N-1}$, one obtains
\begin{eqnarray}\label{MA-3}
\begin{split}
&x_n^m\Delta_{x'}(\eta^{N-1}\partial_{x'}^{N+1}u)+\partial^2_{n}(\eta^{N-1}\partial_{x'}^{N+1}u)\\
=&\quad \eta^{N-1}\partial_{x'}^{N}G+ \sum_{i,j=1}^{n-1}\eta^{N-1}\partial_{x'}^N(a^{ij}x_n^m\partial_{ij}\partial_lu)+[\mathcal L,\eta^{N-1}]\partial_{x'}^{N+1}u \quad \text{in}\quad B_1^+,
\end{split}
\end{eqnarray}
where
\begin{eqnarray}
\begin{split}
[\mathcal L,\eta^{N-1}]\partial_{x'}^{N+1}u&=[x_n^m \Delta_{x'}, \eta^{N-1}]\partial_{x'}^{N+1}u+[\partial_n^2, \eta^{N-1}]\partial_{x'}^{N+1}u\\
                       &=x_n^m (\Delta_{x'} \eta^{N-1})\partial_{x'}^{N+1}u+2x_n^m\sum_{i=1}^{n-1}\partial_i(\eta^{N-1})\partial_i(\partial_x^{N+1}u)\\
                       & \quad\quad +(\partial_n^2\eta^{N-1})\partial_{x'}^{N+1}u+2\partial_n\eta^{N-1}\partial_n\partial_{x'}^{N+1}u.
\end{split}
\end{eqnarray}

As previous, we consider three cases stated in the following three lemmas. In the following proof, the constants $C,C_1...$ may vary from line to line, and can depend on $r$ but is independent of $A_0,A_1$. In addition, we use the constant $c$ to denote the quantities which are independent of $r$. \\

We first prove \eqref{ana1} holds for $\eta^{N-1}\partial_{x'}^{N+1}u$ case.
\begin{lemma}\label{leminduction1}Suppose the assumptions in Theorem \ref{mainthm2} are fulfilled. Suppose \eqref{ana1} holds for sufficiently large $A_0,A_1$. Then there holds
\begin{equation*}
\|\eta^{N-1}\partial_{x'}^{N+1}u\|_{\widetilde W^{k-i,2}} \le C_1\tilde A_0 A_0A_1^{N-4}(N-3-i)!,\quad i=0,1,2
\end{equation*}
for some constant $C_1>0$ and $\tilde A_0$ is the constant in Lemma \ref{lem1}.
\end{lemma}
\begin{proof}
For $i=1$, we know
\begin{equation*}
\begin{split}
\|\eta^{N-1}\partial_{x'}^{N+1}u\|_{\widetilde W^{k-1,2}}=&\|\partial_{x'}(\eta^{N-1}\partial_{x'}^{N}u)-(N-1)\eta^{N-2}\partial_{x'}\eta \partial_{x'}^N u\|_{\widetilde W^{k-1,2}}\\
\le & C_1\|\eta^{N-2}\partial_{x'}^{N}u\|_{\widetilde W^{k,2}}+C_1(N-1)\|\eta^{N-2}\partial_{x'}^{N}u\|_{\widetilde W^{k-1,2}}\\
\le & C A_0 A_1^{N-4}(N-4)!.
\end{split}
\end{equation*}
In getting the last inequality, we used \eqref{ana1}.
By a similar argument, we can prove the case $i=2$.
\par It remains to prove the case $i=0$. Recall that $\eta^{N-1}\partial_{x'}^{N+1}u$ satisfies
\begin{eqnarray}\label{MA-4}
\begin{split}
&x_n^m\Delta_{x'}(\eta^{N-1}\partial_{x'}^{N+1}u)+\partial^2_{n}(\eta^{N-1}\partial_{x'}^{N+1}u)\\
=&\quad \eta^{N-1}\partial_{x'}^{N}G+\sum_{i,j=1}^{n-1}\eta^{N-1}\partial_{x'}^N(a^{ij}x_n^m\partial_{ij}\partial_{x'}u)+[\mathcal L,\eta^{N-1}]\partial_{x'}^{N+1}u\\
=&\quad \mathcal L(\eta^{N-1}\partial_{x'}^{N+1}u) \quad \text{in}\quad \mathbb R^n_+,\\
& \eta^{N-1}\partial_{x'}^{N+1} u(x',0)=\hat \eta^{N-1}\partial_{x'}^{N+1}\varphi(x'),\quad x'\in \mathbb R^{n-1}.
\end{split}
\end{eqnarray}
Here $\hat \eta(x')=\eta(x',0)$.
By Lemma \ref{lemest1}, we know
\begin{equation}
\|\eta^{N-1}\partial_{x'}^{N+1}u\|_{\widetilde W^{k,2}}\le C(\|\mathcal L(\eta^{N-1}\partial_{x'}^{N+1}u)\|_{H^{k}}+\|\hat \eta^{N-1}\partial_{x'}^{N+1}\varphi\|_{H^{k+1}}).
\end{equation}
By the analyticity of $\varphi$, we may assume
\begin{equation*}
\|\hat \eta^{N-1}\partial_{x'}^{N+1}\varphi\|_{H^{k+1}}\le A_0A_1^{N-4}(N-3)!.
\end{equation*}
In the following, we only need to estimate the terms in $\mathcal L(\eta^{N-1}\partial_{x'}^{N+1}u)$. By the definition of $G$, we may denote $G$ as $G=x_n^m \tilde G$ where $\tilde G$ is a polynomial of $\nabla u, \nabla^2 u$.

Firstly, we rewrite $\eta^{N-1}\partial_{x'}^{N}G$ as
\begin{equation*}
\begin{split}
\eta^{N-1}\partial_{x'}^{N}G=&\partial_{x'}^2(x_n^m\eta^{N-1}\partial_{x'}^{N-2}\tilde G)-2(N-1)\partial_{x'}((\partial_{x'} \eta)x_n^m\eta^{N-2}\partial_{x'}^{N-2}\tilde G)\\
&+(N-1)(N-2)(\partial_{x'} \eta)(\partial_{x'} \eta)x_n^m\eta^{N-3}\partial_{x'}^{N-2}\tilde G\\
&+(N-1)(\partial_{x'}^2 \eta)x_n^m\eta^{N-2}\partial_{x'}^{N-2}\tilde G.
\end{split}
\end{equation*}
Then we estimate the terms in $\eta^{N-1}\partial_{x'}^{N}G$ one by one.
\begin{eqnarray}\label{7}
\begin{split}
&\|\partial_{x'}^2(x_n^m\eta^{N-1}\partial_{x'}^{N-2}\tilde G)\|_{H^k}\le \|x_n^m\eta^{N-1}\partial_{x'}^{N-2}\tilde G\|_{H^{k+2}}\\
\le& C \|\eta^{N-2}\partial_{x'}^{N-2}\tilde G\|_{\widetilde W^{k,2}}\le C \tilde A_0 A_1^{N-4}(N-4)!.
\end{split}
\end{eqnarray}
In getting the above inequality, we used Lemma \ref{lem1} and induction assumption \eqref{ana1}.  And also
\begin{equation}
\begin{split}
&\|\partial_{x'}((\partial_{x'} \eta)x_n^m\eta^{N-2}\partial_{x'}^{N-2}\tilde G)\|_{H^k}\le \|(\partial_{x'} \eta)x_n^m\eta^{N-2}\partial_{x'}^{N-2}\tilde G\|_{H^{k+1}}\\
\le& C \|\eta^{N-2}\partial_{x'}^{N-2}\tilde G\|_{\widetilde W^{k-1,2}}\le C \tilde A_0A_1^{N-4}(N-5)!.
\end{split}
\end{equation}
Similar arguments yield that
\begin{eqnarray}\label{8}
\|(\partial_{x'} \eta)(\partial_{x'} \eta)x_n^m\eta^{N-3}\partial_{x'}^{N-2}\tilde G\|_{H^k}\le C \tilde A_0 A_1^{N-4}(N-6)!
\end{eqnarray}
if we notice that $\frac{\eta_j\eta_l}{\eta}\in C_c^\infty(\overline{\mathbb R^n_+})$, $j,l=1,2,\cdots,n-1$.  Combining the estimates \eqref{7}-\eqref{8}, one gets
\begin{equation}
\|\eta^{N-1}\partial_{x'}^{N}G\|_{H^k}\le C\tilde A_0A_1^{N-4}(N-4)!.
\end{equation}

In the following, we estimate $\eta^{N-1}\partial_{x'}^N(a^{ij}x_n^m\partial_{ij}\partial_{x'}u)$. As previous, we rewrite this term as
\begin{equation}
\begin{split}
&\eta^{N-1}\partial_{x'}^N(a^{ij}x_n^m\partial_{ij}\partial_{x'}u)\\
=&\partial_{x'}^2(x_n^m\eta^{N-1}\partial_{x'}^{N-2}(a^{ij}\partial_{ij}\partial_{x'}u))-2(N-1)\partial_{x'}((\partial_{x'} \eta)x_n^m\eta^{N-2}\partial_{x'}^{N-2}(a^{ij}\partial_{ij}\partial_{x'}u))\\
&+(N-1)(N-2)(\partial_{x'} \eta)(\partial_{x'} \eta)x_n^m\eta^{N-3}\partial_{x'}^{N-2}(a^{ij}\partial_{ij}\partial_{x'}u)\\
&+(N-1)(\partial_{x'}^2 \eta)x_n^m\eta^{N-2}\partial_{x'}^{N-2}(a^{ij}\partial_{ij}\partial_{x'}u).
\end{split}
\end{equation}
Then one knows
\begin{eqnarray}\label{est3}
\begin{split}
&\|\partial_{x'}^2(x_n^m\eta^{N-1}\partial_{x'}^{N-2}(a^{ij}\partial_{ij}\partial_{x'}u))\|_{H^k}\le \|\eta^{N-1}\partial_{x'}^{N-2}(a^{ij}\partial_{ij}\partial_{x'}u)\|_{\widetilde W^{k,2}}\\
\le & cr\|\eta^{N-1}\partial_{x'}^{N+1} u\|_{{\widetilde W^{k,2}}}+\sum_{l= 1}^{N-2}\frac{(N-2)!}{l!(N-2-l)!}\|\eta^{N-1}\partial_{x'}^la^{ij}\partial_{x'}^{N-2-l}\partial_{ij}\partial_{x'}u\|_{\widetilde W^{k,2}}\\
\le &cr\|\eta^{N-1}\partial_{x'}^{N+1} u\|_{{\widetilde W^{k,2}}}+C \tilde A_0 A_0 A_1^{N-4}(N-3)!.
\end{split}
\end{eqnarray}
The constant $c$ in the above inequality is independent of $r$.
In getting the last inequality of \eqref{est3}, we need to use the induction assumption \eqref{ana1} and  do calculations as in Lemma \ref{lem1}.
Since only the order of derivative in  $\eta^{N-1}\partial_{x'}^la^{ij}\partial_{x'}^{N-2-l}\partial_{ij}\partial_{x'}u$ matters to the estimate, we make a convention that all the terms $\eta^{N-1}\partial_{x'}^l a^{ij}\partial_{x'}^{N-2-l}\partial_{ij}\partial_{x'}u$ are the same for a fixed $l$.

Similarly, we estimate the other terms in $\eta^{N-1}\partial_{x'}^N(a^{ij}x_n^m\partial_{ij}\partial_{x'}u)$ and then get
\begin{equation}
\|\eta^{N-1}\partial_{x'}^N(a^{ij}x_n^m\partial_{ij}\partial_{x'}u)\|_{H^k}\le cr\|\eta^{N-1}\partial_{x'}^{N+1} u\|_{{\widetilde W^{k,2}}}+C \tilde A_0 A_0 A_1^{N-4}(N-3)!.
\end{equation}
\par For the last term $[\mathcal L,\eta^{N-1}]\partial_{x'}^{N+1}u$, the part $[x_n^m\Delta_{x'},\eta^{N-1}]\partial_{x'}^{N+1}u$ always contains the factor $x_n^m$, we can deduce the $H^k$-norm for $[x_n^m\Delta_{x'},\eta^{N-1}]\partial_{x'}^{N+1}u$ exactly the same as previous two terms in $\mathcal L(\eta^{N-1}\partial_{x'}^{N+1}u)$.

In the following, we estimate the remaining term $[\partial_n^2,\eta^{N-1}]\partial_{x'}^{N+1}u$.
\begin{eqnarray}\label{anainterior2}
\begin{split}
&[\partial_n^2,\eta^{N-1}]\partial_{x'}^{N+1}u=2(N-1)\eta_n\eta^{N-2}\partial_n\partial_{x'}^{N+1}u\\
&\quad\quad +(N-1)(N-2)\eta_n^2 \eta^{N-3}\partial_{x'}^{N+1}u+(N-1)\eta_{nn}\eta^{N-2}\partial_{x'}^{N+1}u.
\end{split}
\end{eqnarray}
By the definition of $\eta$, we know
\begin{equation*}
\text{supp }\eta_n\subset [-2r,2r]^{n-1}\times[r,2r](:=Q_r).
\end{equation*}

From the discussion at the beginning of this section, we know $u$ is analytic in $\{x_n>0\}$.  Hence, in the following, we can always assume
\begin{eqnarray}\label{anainterior1}
\|\eta^{\tilde N-2}\partial_x^{\tilde N+2} u\|_{H^k(Q_r)}\le A_0 A_1^{(\tilde N-4)^+}(\tilde N-4)^+!,\quad \forall \tilde N=0,1,2,\cdots.
\end{eqnarray}
Then from \eqref{anainterior2} and \eqref{anainterior1}, one obtains
\begin{equation}
\|[\partial_n^2,\eta^{N-1}]\partial_{x'}^{N+1}u\|_{H^k}\le C A_0 A_1^{N-4}(N-3)!.
\end{equation}

Combining all the above estimates, one gets
\begin{equation*}
\|\eta^{N-1}\partial_{x'}^{N+1}u\|_{\widetilde W^{k,2}}\le cr\|\eta^{N-1}\partial_{x'}^{N+1}u\|_{\widetilde W^{k,2}}+C \tilde A_0 A_0 A_1^{N-4}(N-3)!.
\end{equation*}
By taking $r$ small enough, one proves the present lemma.
\end{proof}

\begin{lemma}\label{leminduction2}
Suppose the assumptions in Theorem \ref{mainthm2} are fulfilled. Suppose \eqref{ana1} holds for sufficiently large $A_0,A_1$. Then there holds
\begin{equation*}
\|\eta^{N-1}\partial_n\partial_{x'}^{N}u\|_{\widetilde W^{k-i,2}} \le C_1\tilde A_0 A_0 A_1^{N-4}(N-3-i)!,\quad i=0,1,2.
\end{equation*}
\end{lemma}
\begin{proof}
As the proof in Lemma \ref{leminduction1}, we only need to show it for $i=0$.
Recall that $\eta^{N-1}\partial_n\partial_{x'}^{N}u$ satisfies
\begin{eqnarray}\label{MA-5}
\begin{split}
&x_n^m\Delta_{x'}(\eta^{N-1}\partial_n\partial_{x'}^{N}u)+\partial^2_{n}(\eta^{N-1}\partial_n\partial_{x'}^{N}u)\\
=&\quad \eta^{N-1}\partial_n\partial_{x'}^{N-1}G+\sum_{i,j=1}^{n-1}\eta^{N-1}\partial_n\partial_{x'}^{N-1}(a^{ij}x_n^m\partial_{ij}\partial_{x'}u)\\
& \quad \quad +[\mathcal L,\eta^{N-1}]\partial_n\partial_{x'}^{N}u-mx_n^{m-1}\eta^{N-1}\Delta_{x'}\partial_{x'}^N u \quad \text{in}\quad \mathbb R^n_+.
\end{split}
\end{eqnarray}
We first consider the boundary term $\eta^{N-1}\partial_n\partial_{x'}^{N}u(x',0)$.
\begin{eqnarray}\label{9}
\begin{split}
&\|\eta^{N-1}\partial_n\partial_{x'}^{N}u(x',0)\|_{H^{k+1}}\\
= & \|\partial_n\partial_{x'}(\eta^{N-1}\partial_{x'}^{N}u(x',0))\|_{H^{k}}\\ \le  &\|\partial_n(\eta^{N-1}\partial_{x'}^{N+1}u(x',0))\|_{H^k}+(N-1)\|\eta_{x'}\partial_n(\eta^{N-2}\partial_{x'}^{N}u(x',0))\|_{H^k}\\
\le & C \tilde A_0 A_0 A_1^{N-4}(N-3)!.
\end{split}
\end{eqnarray}
In the calculation of \eqref{9}, the property of $\eta(x)=\chi(x_1)\cdots\chi(x_n)$ is used in the commutation of $\partial_n$ and $\eta^{N-1}$ on $x_n=0$.
In getting the last inequality of \eqref{9}, we used Lemma \ref{leminduction1}.
\par In the following,  we estimate the $H^k$-norm for $\eta^{N-1}\partial_n\partial_{x'}^{N-1}G, \eta^{N-1}\partial_n\partial_{x'}^{N-1}(a^{ij}x_n^m\partial_{ij}\partial_{x'}u)$, $ [\mathcal L,\eta^{N-1}]\partial_n\partial_{x'}^{N}u, mx_n^{m-1}\eta^{N-1}\Delta_{x'}\partial_{x'}^N u $ due to Lemma \ref{lemest1} and \eqref{MA-5}.

As in the proof of Lemma \ref{leminduction1}, we rewrite $\eta^{N-1}\partial_n \partial_{x'}^{N-1} G$ as
\begin{eqnarray}\label{bb1}
\begin{split}
\eta^{N-1}\partial_n \partial_{x'}^{N-1} G=&\eta^{N-1}\partial_n(x_n^m \partial_{x'}^{N-1}\tilde G)\\=&\partial_{nx'}(\eta^{N-1}x_n^m \partial_{x'}^{N-2}\tilde G)-(N-1)\partial_n(\eta_{x'}\eta^{N-2}x_n^m\partial_{x'}^{N-2}\tilde G)\\&-(N-1)\partial_{x'}(\eta_n\eta^{N-2}x_n^m\partial_{x'}^{N-2}\tilde G)\\
&+(N-1)(N-2)\eta_n\eta_{x'}\eta^{N-3}x_n^m\partial_{x'}^{N-2}\tilde G+(N-1)\eta_{nx'}\eta^{N-2}x_n^m\partial_{x'}^{N-2}\tilde G.
\end{split}
\end{eqnarray}
Then all the terms on the right hand-side of \eqref{bb1}  can be estimated as in the proof of Lemma \ref{leminduction1} due to the presence of $x_n^m$.  Thus we have
\begin{equation*}
\|\eta^{N-1}\partial_n(x_n^m \partial_{x'}^{N-1}\tilde G)\|_{H^k}\le C\tilde A_0 A_1^{N-4}(N-4)!.
\end{equation*}

Similarly, $\eta^{N-1}\partial_n\partial_{x'}^{N-1}(a^{ij}x_n^m\partial_{ij}u_{x'})$, $ [\mathcal L,\eta^{N-1}]\partial_n\partial_{x'}^{N}u$ can be estimated exactly as in Lemma \ref{leminduction1}. We only need to take care of $x_n^{m-1}\eta^{N-1}\Delta_{x'}\partial_{x'}^N u$.
\begin{eqnarray}\label{bb2}
\begin{split}
&x_n^{m-1}\eta^{N-1}\Delta_{x'}\partial_{x'}^N u
\\=& \partial_{x'}(x_n^{m-1}\eta^{N-1}\partial_{x'}^{N-1}\Delta_{x'}u)-(N-1)x_n^{m-1}\partial_{x'}(\eta_{x'}\eta^{N-2}\partial_{x'}^{N-2}\Delta_{x'}u)\\
+&(N-1)(N-2)x_{n}^{m-1}\eta_{x'}\eta_{x'}\eta^{N-3}\partial_{x'}^{N-2}\Delta_{x'} u+(N-1)\eta_{x'x'}\eta^{N-2}x_n^{m-1}\partial_{x'}^{N-2}\Delta_{x'}u.
\end{split}
\end{eqnarray}
For the first term $\partial_{x'}(x_n^{m-1}\eta^{N-1}\partial_{x'}^{N-1}\Delta_{x'}u)$, we have
\begin{equation}
\begin{split}
&\|\partial_{x'}(x_n^{m-1}\eta^{N-1}\partial_{x'}^{N-1}\Delta_{x'}u)\|_{H^k}\le \|x_n^{m-1}\eta^{N-1}\partial_{x'}^{N-1}\Delta_{x'}u\|_{H^{k+1}}\\
\le &C_1\|\eta^{N-1}\partial_{x'}^{N+1} u\|_{\widetilde W^{k,2}}\le C_2\tilde A_0 A_0 A_1^{N-4}(N-3)!.
\end{split}
\end{equation}
In getting the above inequality, we used Lemma \ref{leminduction1} and the definition of $\widetilde W^{k,2}$. All the remaining terms in \eqref{bb2} can be estimated in a similar way.

Combining all the above estimates together yield the present lemma.
\end{proof}

\begin{lemma}\label{leminduction3}Suppose the assumptions in Theorem \ref{mainthm2} are fulfilled. Suppose \eqref{ana1} holds for sufficiently large $A_0,A_1$.
For $2 \le l\le N+1$, there holds
\begin{equation*}
\|\eta^{N-1}\partial^l_n\partial_{x'}^{N+1-l}u\|_{\widetilde W^{k-i,2}} \le C_1\tilde A_0 A_0 A_1^{N-4}(N-3-i)!,\quad i=0,1,2.
\end{equation*}
\end{lemma}
\begin{proof}
As previous, we only need to show the case for $i=0$.
Recall that $\eta^{N-1}\partial^l_n\partial_{x'}^{N+1-l}u$ satisfies
\begin{eqnarray}\label{MA-6}
\begin{split}
&x_n^m\Delta_{x'}(\eta^{N-1}\partial^l_n\partial_{x'}^{N+1-l}u)+\partial^2_{n}(\eta^{N-1}\partial^l_n\partial_{x'}^{N+1-l}u)\\
=&\quad \eta^{N-1}\partial^{l-1}_n\partial_{x'}^{N+1-l}G+\sum_{i,j=1}^{n-1}\eta^{N-1}\partial^{l-1}_n\partial_{x'}^{N-1+l}(a^{ij}x_n^m\partial_{ij}\partial_n u)+[\mathcal L,\eta^{N-1}]\partial^l_n\partial_{x'}^{N+1-l}u\\
&\quad -\eta^{N-1}\sum_{l'=1}^{\min(l-1,m)}\frac{(l-1)!}{l'!(l-1-l')!}\partial_n^{l'}(x_n^m)\partial_{n}^{l-l'-1}\partial_{x'}^{N+1-l}(\Delta_{x'}\partial_n u) \quad \text{in}\quad \mathbb R^n_+
\end{split}
\end{eqnarray}
where
\begin{eqnarray*}
[\mathcal L,\eta^{N-1}]\partial^l_n\partial_{x'}^{N+1-l}u= [x_n^m\Delta_{x'},\eta^{N-1}]\partial^l_n\partial_{x'}^{N+1-l}u+[\partial_n^2,\eta^{N-1}]\partial^l_n\partial_{x'}^{N+1-l}u.
\end{eqnarray*}
For the boundary term $\eta^{N-1}\partial^l_n\partial_{x'}^{N+1-l}u(x',0)$, we distinguish with two cases:
\begin{itemize}
\item[(1)] $2\le l\le m+1$. Then by the equation \eqref{MA-1}, we know $\eta^{N-1}\partial^l_n\partial_{x'}^{N+1-l}u(x',0)\equiv 0$.
\item[(2)] $m+2\le l\le N+1$. Differentiating \eqref{MA-1} with respect to $x_n$ for $l-2$ times and then with respect to $x'$ for $N+1-l$ times, one gets
 \begin{equation}
\begin{split}
&\eta^{N-1}\partial^l_n\partial_{x'}^{N+1-l}u(x',0)\\
=&-\eta^{N-1}\partial_{n}^{l-2}\partial_{x'}^{N+1-l}(x_n^m(-u_n)^{n+2}\det D^2_{x'}u)\bigg |_{x_n=0}\\=&-\eta^{N-1}\frac{(l-2)!}{(l-2-m)!}\partial^{l-2-m}_n\partial_{x'}^{N+1-l}(x_n^m(-u_n)^{n+2}\det D^2_{x'}u)(x',0).
\end{split}
\end{equation}
\end{itemize}
Then following the same calculations as in Lemma \ref{lem1} and using the induction assumptions, one obtains
\begin{equation}
\|\eta^{N-1}\partial^l_n\partial_{x'}^{N+1-l}u(x',0)\|_{H^{k+1}}\le C\tilde A_0 A_1^{N-4}(N-4)!.
\end{equation}

For \eqref{MA-6}, we can deduce the $H^k$-norm for $\eta^{N-1}\partial^{l-1}_n\partial_{x'}^{N+1-l}G$, $[\mathcal L,\eta^{N-1}]\partial^l_n\partial_{x'}^{N+1-l}u$ and $\eta^{N-1}\partial^{l-1}_n\partial_{x'}^{N-1+l}(a^{ij}x_n^m\partial_{ij}\partial_n u)$ exactly the same as previous lemmas.

Afterwards, we only need to take care of the term $\eta^{N-1}\partial_n^{l'}(x_n^m)\partial_{n}^{l-l'-1}\partial_{x'}^{N+1-l}(\Delta_{x'}\partial_n u)$.  This term can be discussed in the following different cases:
\begin{itemize}
\item[(1)] $l=2, l'=1$.  We have
\begin{equation*}
\eta^{N-1}\partial_n^{l'}(x_n^m)\partial_{n}^{l-l'-1}\partial_{x'}^{N+1-l}(\Delta_{x'}\partial_n u)=mx_n^{m-1}\eta^{N-1}\partial_{x'}^{N-1}\Delta_{x'}\partial_n u.
\end{equation*}
Then similar arguments as \eqref{bb2} and by the results of Lemma \ref{leminduction2}, one knows
\begin{equation*}
\|\eta^{N-1}x_n^{m-1}\partial_{x'}^{N-1}\Delta_{x'}\partial_n u\|_{H^k}\le C\tilde A_0 A_0 A_1^{N-4}(N-3)!.
\end{equation*}
\item[(2)] $l\ge 3$, $1\le l'\le\min( l-1,m)$. We can rewrite $\eta^{N-1}\partial_{n}^{l-l'-1}\partial_{x'}^{N+1-l}(\Delta_{x'}\partial_n u)$ as
\begin{equation*}
\begin{split}
&\eta^{N-1}\partial_{n}^{l-l'-1}\partial_{x'}^{N+1-l}(\Delta_{x'}\partial_n u)=\partial_{n}^2(\eta^{N-1}\partial_n^{l-l'-2}\partial_{x'}^{N+1-l}\Delta_{x'}u)\\
-&2(N-1)\partial_n(\eta_n\eta^{N-2}\partial_n^{l-l'-2}\partial_{x'}^{N+1-l}\Delta_{x'}u)+(N-1)\eta_{nn}\eta^{N-2}\partial_n^{l-l'-2}\partial_{x'}^{N+1-l}\Delta_{x'}u
\\+&(N-1)(N-2)\eta_n^2 \eta^{N-3}\partial_n^{l-l'-2}\partial_{x'}^{N+1-l}\Delta_{x'}u.
\end{split}
\end{equation*}
By the induction assumption \eqref{ana1}, we know
\begin{equation*}
\begin{split}
&\|\partial_{n}^2(\eta^{N-1}\partial_n^{l-l'-2}\partial_{x'}^{N+1-l}\Delta_{x'}u)\|_{H^k}\\
\le &\|\eta^{N-1}\partial_n^{l-l'-2}\partial_{x'}^{N+1-l}\Delta_{x'}u)\|_{\widetilde W^{k,2}}\\
\le &A_0 A_1^{(N-3-l')^+}(N-3-l')^+!.
\end{split}
\end{equation*}
Similar estimates also hold for the remaining terms of $\eta^{N-1}\partial_{n}^{l-l'-1}\partial_{x'}^{N+1-l}(\Delta_{x'}\partial_nu)$ by noticing that these terms contain $\eta_n$.

\quad \quad  Then
\begin{equation*}
\|\eta^{N-1}\partial_{n}^{l-l'-1}\partial_{x'}^{N+1-l}(\Delta_{x'}\partial_n u)\|_{H^k}\leq C A_0 A_1^{(N-3-l')^+}(N-3-l')^+!.
\end{equation*}

This implies
\begin{equation}
\begin{split}
&\sum_{l'=1}^{\min(l-1,m)}\frac{(l-1)!}{l'!(l-1-l')!}\|\eta^{N-1}\partial_n^{l'}(x_n^m)\partial_{n}^{l-l'-1}\partial_{x'}^{N+1-l}(\Delta_{x'}\partial_{n})\|_{H^k}\\
\le & C A_0 A_1^{N-4}(N-3)!.
\end{split}
\end{equation}
\end{itemize}
Thus, applying Lemma \ref{lemest1}, one proves the present lemma.
\end{proof}
\begin{remark}
Combining the estimates of Lemma \ref{leminduction1}-\ref{leminduction3} and choosing $A_1\ge C_1\tilde A_0 $ large enough, we can get the estimate \eqref{ana1} for $N+1$.
\end{remark}

\section{Acknowledgments}
The work of the first author is sponsored by Shanghai Rising-Star Program 19QA1400900.

\bibliographystyle{abbrv}
\bibliography{1_mainbib-hgg}

\end{document}